\documentclass[11pt]{amsart}    

\usepackage[pagebackref=true,colorlinks,linkcolor=blue,citecolor=magenta]{hyperref}
\usepackage{amsmath,amssymb,latexsym} 
\usepackage{graphicx}\usepackage{centernot}

\usepackage{fullpage}
\usepackage[square, numbers]{natbib}   
\usepackage{comment}

\usepackage{enumitem}

\newtheorem{theorem}{Theorem} 

\newtheorem{alphtheorem}{Theorem}

\newtheorem{alphconjecture}{Conjecture}

\newtheorem{corollary}{Corollary}

\newtheorem{lemma}{Lemma}
\theoremstyle{definition}

\theoremstyle{remark}

\def\zero{\boldsymbol{0}}

\def\ds{\displaystyle}
\def\KG{\operatorname{KG}}
\def\SG{\operatorname{SG}}
\def\cd{\operatorname{cd}}

\def\alt{\operatorname{alt}}
\def\salt{\operatorname{salt}}

\title{A Generalization of Gale's lemma}
\author{Meysam Alishahi}
\address{M. Alishahi, 
School of Mathematical Sciences,
Shahrood University of Technology, Shahrood, Iran}
\email{meysam\_alishahi@shahroodut.ac.ir}

\author{Hossein Hajiabolhassan}
\address{H. Hajiabolhassan, Department of Applied Mathematics and Computer Science,Technical University of Denmark, DK-{\rm 2800} Lyngby, Denmark, and
Department of Mathematical Sciences, Shahid Beheshti University, G.C., P.O. Box {\rm 19839-69411}, Tehran, Iran}
\email{hhaji@sbu.ac.ir}

\begin{document}
\maketitle

\begin{abstract}
In this work, we present a generalization of Gale's lemma. Using this generalization, we introduce two combinatorial sharp lower bounds for ${\rm conid}({\rm B}_0(G))+1$ and ${\rm conid}({\rm B}(G))+2$,  two famous topological lower bounds for the chromatic number of a graph $G$. 

\end{abstract}
\section{Introduction and Main Results}
\subsection{Backgrounds and Motivations}
Throughout the paper, 
for positive integers $k$ and $n$, two symbols $[n]$ and ${[n]\choose k}$
stand for the set $\{1,\ldots,n\}$ and the family of all $k$-subsets of $[n]$, respectively.
For a positive integer $s$, 
a subset $A$ of $[n]$ is said to be an $s$-stable 
subset if $s\leq |i-j|\leq n-s$ for each $i\neq j\in A$. Throughout the paper, the family of all $s$-stable $k$-subsets of $[n]$ is denoted by ${[n]\choose k}_{s}$.
For $n\geq 2k$, the Kneser graph $\KG(n,k)$ is a graph with vertex set consisting of all $k$-subsets of $[n]$ as vertex set and two vertices are adjacent if their corresponding 
$k$-sets are disjoint. Kneser~1955~\cite{MR0068536} proved that $\KG(n,k)$ can be properly colored with $n-2k+2$ colors. He also conjectured that this is the best possible, i.e., $\chi(\KG(n,k))\geq n-2k+2$. In~1978, Lov\'asz in a fascinating paper~\cite{MR514625}, using the Borsuk-Ulam theorem, proved this conjecture.

For an integer $d\geq -1$, by the symbol $S^d$, we mean the $d$-dimensional sphere. 
For an $x\in S^d$, $H(x)$ is the open hemisphere centered at $x$, i.e., 
$H(x)=\{y\in S^d\;:\; \langle x\, ,y \rangle>0\}$. 
There is a famous lemma due to Gale~\cite{MR0085552} which asserts  
that for every $k\geq 1$ and every $n\geq 2k$, there is an $n$-set $Z\subset S^{n-2k}$
such that for any $x\in S^{n-2k}$, the open hemisphere $H(x)$ contains at least $k$ points of $Z$. In particular,
if we identify the set $Z$ with $[n]$, then  for any $x\in S^{n-2k}$, the open hemisphere $H(x)$ contains some vertex of ${\rm KG}(n,k)$.
Soon after the announcement of the Lov\'asz breakthrough~\cite{MR514625}, B{\'a}r{\'a}ny~\cite{MR514626} presented a
short proof of the Lov\'asz-Kneser theorem based on Gale's lemma. 
Next, Schrijver~\cite{MR512648} generalized Gale's lemma  by proving that there is an $n$-set $Z\subset S^{n-2k}$ and a suitable identification of $Z$ with $[n]$
such that for any $x\in S^{n-2k}$, the open hemisphere $H(x)$ contains a $2$-stable subset of size at least $k$, i.e.,
$H(x)$ contains some vertex of ${\rm SG}(n,k)$.  Using this generalization, Schrijver~\cite{MR512648} found a vertex-critical subgraph $\SG(n,k)$ of $\KG(n,k)$ which has the same chromatic number as $\KG(n,k)$. To be more specific, the Schrijver graph $\SG(n,k)$ is an induced subgraph of $\KG(n,k)$ whose vertex set consists of all $2$-stable $k$-subsets of $[n]$. 

\subsection{Main Results}
For an $X=(x_1,\ldots,x_n)\in\{+,-,0\}^n$, an alternating subsequence of $X$ is a subsequence of nonzero terms of $X$ such that each of its two consecutive members have  different signs. In other words, $x_{j_1},\ldots,x_{j_m}$ ($1\leq j_1<\cdots<j_m\leq n$) is an alternating subsequence of $X$ if $x_{j_i}\neq 0$ for each $i\in[m]$ and  $x_{j_i}\neq x_{j_{i+1}}$ for  $i=1,\ldots, m-1$. The length of the longest alternating subsequence of $X$ is denoted by $\alt(X)$. We also define $\alt(0,\ldots,0)=0$. Moreover, 
define 
$$X^+=\{j:\; x_j=+\}\quad\mbox{ and }\quad X^-=\{j:\; x_j=-\}.$$
Note that, by abuse of notation, we can write $X=(X^+,X^-)$.

Let $V$ be a nonempty finite set of size $n$.
The signed-power set of $V$, denoted $P_s(V)$, is defined as follows;
$$P_s(V)=\left\{(A,B)\ :\ A,B\subseteq V,\, A\cap B=\varnothing\right\}.$$
For two pairs $(A,B)$ and $(C,D)$ in $P_s(V)$, by $(A,B)\subseteq (C,D)$, we mean
$A\subseteq C$ and $B\subseteq D$. 
Note that $\left(P_s(V),\subseteq\right)$ is a partial order set (poset). 
A signed-increasing property ${\mathcal P}$, is a superset-closed  family ${\mathcal P}\subseteq P_s(V)$, i.e., for any $F_1 \in {\mathcal P}$, if $F_1\subseteq F_2\in P_s(V)$, then $F_2\in {\mathcal P}$. 
Clearly, for any  bijection $\sigma:[n]\longrightarrow V$, the map $X\mapsto X_\sigma=(\sigma(X^+),\sigma(X^-)$ is an identification between $\{+,-,0\}^n$ and $P_s(V)$. 
Let $\sigma:[n]\longrightarrow V$ be a bijection and ${\mathcal P}\subseteq P_s(V)$
be a signed-increasing  property. Define 
$$\alt({\mathcal P},\sigma)=\max\left\{\alt(X):\ X\in\{+,-,0\}^{n} \mbox{ with } \ {X_\sigma}\not\in{\mathcal P}\right\}.$$
Also, define {\it the alternation number of ${\mathcal P}$} to be the following quantity;
$$\alt({\mathcal P})=\ds\min\{\alt({\mathcal P},\sigma): \; \sigma:[n]\longrightarrow V\mbox{ is a bijection}\}.$$

Let $d\geq 0$ be an integer, $S^d$ be the $d$-dimensional sphere,  and $Z\subset S^d$ be a finite set.
For an $x\in S^d$, define $Z_x=(Z_x^+,Z_x^-)\in P_s(Z)$ such that $Z_x^+=H(x)\cap Z$ and $Z_x^-=H(-x)\cap Z$. 
Now, we are in a position to state the first main result of this paper, which is a generalization of  Gale's lemma.
\begin{lemma}\label{galegen}
Let $n$ be a positive integer, $V$ be an $n$-set, and $\sigma:[n]\longrightarrow V$ be a bijection.
Also, let ${\mathcal P}\subseteq P_s(V)$ be a signed-increasing property and set
$d=n-alt({\mathcal P},\sigma)-1$. If $d\neq -1$, then there are a multiset $Z\subset S^{d}$ of size $n$ and a suitable identification of $Z$ with $V$ such that for any $x\in S^{d}$,  $Z_x\in{\mathcal P}$. In particular, for $d \geq 1$, $Z$ can be a set.
\end{lemma}

A hypergraph $\mathcal{H}$ is a pair $(V(\mathcal{H}),E(\mathcal{H}))$ where $V(\mathcal{H})$ is a finite nonempty set, called the vertex set of $\mathcal{H}$, and $E(\mathcal{H})$ is a family  containing some nonempty distinct subsets of $V(\mathcal{H})$, called the edge set of $\mathcal{H}$. For a set $U\subseteq V(\mathcal{H})$, the induced  
subhypergraph $\mathcal{H}[U]$ of $\mathcal{H}$, is a 
hypergraph  with vertex set $U$ and edge set $\{e\in E(\mathcal{H})\;:\; e\subseteq U\}$. 
A graph $G$ is a hypergraph such that each of its edges has cardinality two.
A $t$-coloring of a hypergraph $\mathcal{H}$  is a map $c:V(\mathcal{H})\longrightarrow [t]$
such that for no edge $e\in E(\mathcal{H})$, we have $|c(e)|=1$.
The minimum possible $t$ for which $\mathcal{H}$ admits a $t$-coloring, denoted $\chi(\mathcal{H})$, is called the chromatic number of $\mathcal{H}$. 
For a hypergraph $\mathcal{H}$, the Kneser graph $\KG(\mathcal{H})$ is a graph whose vertex set  is $E(\mathcal{H})$ and two vertices are adjacent if their corresponding edges are vertex disjoint.
It is known that for any graph $G$, there are several hypergraphs $\mathcal{H}$ such that
$G$ and $\KG(\mathcal{H})$ are isomorphic. Each of such hypergraphs $\mathcal{H}$ is called a Kneser representation of $G$.
Note that if we set $K_n^k=\left([n],{[n]\choose k}\right)$ and $\widetilde{K_n^k}=\left([n],{[n]\choose k}_{2}\right)$, then $\KG(K_n^k)=\KG(n,k)$ and $\KG(\widetilde{K_n^k})=\SG(n,k)$.
The {\it colorability defect} of a hypergraph $\mathcal{H}$, denoted $\cd(\mathcal{H})$, 
is the minimum number of vertices which should be excluded so that the induced subhypergraph on the remaining vertices is $2$-colorable.
Dol'nikov~\cite{MR953021} improved Lov\'asz's result~\cite{MR514625} by proving that for any hypergraph $\mathcal{H}$, we have
$\chi(\KG(\mathcal{H}))\geq\cd(\mathcal{H}).$

Let ${\mathcal H}=(V,E)$ be a hypergraph and $\sigma:[n]\longrightarrow V(\mathcal{H})$ 
be a bijection. Define
$$\alt(\mathcal{H},\sigma)=\max\left\{\alt(X):\; X\in\{+,-,0\}^n\mbox{ s.t. } \max\left(|E(\mathcal{H}[\sigma(X^+)])|,|E(\mathcal{H}[\sigma(X^-)])|\right)=0\right\}$$ 
and
$$\salt(\mathcal{H},\sigma)=\max\left\{\alt(X):\; X\in\{+,-,0\}^n\mbox{ s.t. } \ds\min\left(\left|E(\mathcal{H}[\sigma(X^+)])\right|,|E(\mathcal{H}[\sigma(X^-)])|\right)=0\right\}.$$ 
In other words, $\alt(\mathcal{H},\sigma)$ (resp. $\salt(\mathcal{H},\sigma)$) is the maximum possible $\alt(X)$, where $X\in\{+,-,0\}^n$, such that each of (resp. at least one of) $\sigma(X^+)$ and $\sigma(X^+)$ contains no edge of $\mathcal{H}$.	

Also, we define
$$\alt(\mathcal{H})=\ds\min_{\sigma}\alt(\mathcal{H},\sigma)\quad \mbox{and} \quad
\salt(\mathcal{H})=\ds\min_{\sigma}\salt(\mathcal{H},\sigma),$$
where the minimum is taken over all bijections $\sigma:[n]\longrightarrow V(\mathcal{H})$.
The present authors, using Tucker's lemma~\cite{MR0020254}, introduced  two combinatorial tight lower bounds for the chromatic number of $\KG(\mathcal{H})$ improving  Dol'nikov's lower bound.
\begin{alphtheorem}{\rm \cite{2013arXiv1302.5394A}}\label{alihajijctb}
For any hypergraph $\mathcal{H}$, we have
$$\chi(\KG(\mathcal{H}))\geq\ds\max\left(|V(\mathcal{H})|-\alt(\mathcal{H}),|V(\mathcal{H})|-\salt(\mathcal{H})+1\right).$$
\end{alphtheorem}

Let ${\mathcal H}=(V,E)$ be a hypergraph and $\sigma:[n]\longrightarrow V(\mathcal{H})$ 
be a bijection.
Note that if we set
$${\mathcal P}_1=\left\{(A,B)\in P_s(V):\mbox{ at least one of $A$ and $B$ contains  some edge of } {\mathcal H}\right\}$$
and
$${\mathcal P}_2=\left\{(A,B)\in P_s(V):\mbox{ both of $A$ and $B$ contain  some edges of } {\mathcal H}\right\},$$
then  $\alt({\mathcal P}_1,\sigma)=\alt({\mathcal H},\sigma)$ and  $\alt({\mathcal P}_2,\sigma)=\salt({\mathcal H},\sigma)$. Therefore, in view of Lemma~\ref{galegen}, 
we have the next result.
\begin{corollary}\label{galegencor}
For a hypergraph ${\mathcal H}=(V, E)$ and a bijection $\sigma:[n]\longrightarrow V(\mathcal{H})$,
we have the following assertions.
\begin{itemize}
\item[{\rm a)}]If $d=|V|-\alt({\mathcal H},\sigma)-1$ and $\alt({\mathcal H},\sigma)\neq |V|$, then there are a  multiset $Z\subset S^{d}$ of size $|V|$ and a suitable identification of $Z$ with $V$
such that for any $x\in S^{d}$, $H(x)$ or $H(-x)$ contains some edge of ${\mathcal H}$. 
In particular, for $d\geq 1$, $Z$ can be a set.
\item[{\rm b)}]If $d=|V|-\salt({\mathcal H},\sigma)-1$ and $\salt({\mathcal H},\sigma)\neq |V|$, then there are a multiset $Z\subset S^{d}$ of size $|V|$ and a suitable identification of $Z$ with $V$
such that for any $x\in S^{d}$, $H(x)$ contains some edge of ${\mathcal H}$.
In particular, for $d\geq 1$, $Z$ can be a set.
\end{itemize}
\end{corollary}
For two positive integers $n$ and $k$, where $n> 2k$, 
and for the identity bijection $I:[n]\longrightarrow [n]$, one can see that 
$\salt(\widetilde{K_n^k},I)=2k-1$. Therefore, by the second part of 
Corollary~\ref{galegencor}, we have the generalization of 
Gale's lemma given by Schrijver~\cite{MR512648}: 
{\it there is an $n$-subset $Z$ of $S^{n-2k}$ and a suitable identification 
of $Z$ with $[n]$ such that
for any $x\in S^{n-2k}$, the hemisphere $H(x)$ contains at least a $2$-stable subset of $[n]$ with size at least $k$.}

Note that for any graph $G$, there are several hypergraphs $\mathcal{H}$ 
such that $\KG(\mathcal{H})$ and $G$ are isomorphic. 
By the help of Lemma~\ref{galegen} and as the second main result of this paper, 
we provide two combinatorial approximations for two important topological 
lower bounds for the chromatic number of a graph $G$, namely, $ {\rm coind}(B_0(G))+1$ and ${\rm coind}(B(G))+2$, see~\cite{MR2452828}. The quantities ${\rm coind}(B_0(G))$ and ${\rm coind}(B(G))$ are respectively the coindices    
of two box complexes ${\rm B}_0(G)$ and ${\rm B}(G)$ which 
will be defined in Section~\ref{Preliminaries}.  It should be mentioned that 
the two inequalities $\chi(G) \geq {\rm coind}(B_0(G))+1$ 
$\chi(G) \geq {\rm coind}(B(G))+2$ 
are already proved~\cite{MR2452828} and we restate them in the following 
theorem just to emphasize that this theorem is 
an improvement of Theorem~\ref{alihajijctb}. 
\begin{theorem}\label{coind}
Let $G$ be a graph and $\mathcal{H}$ be a hypergraph such that $\KG(\mathcal{H})$ and 
$G$ are isomorphic. Then the following inequalities hold;
\begin{itemize}
\item[{\rm a)}]  $\chi(G) \geq {\rm coind}(B_0(G))+1\geq |V(\mathcal{H})|-\alt(\mathcal{H})$,
\item[{\rm b)}]  $\chi(G) \geq {\rm coind}(B(G))+2\geq  |V(\mathcal{H})|-\salt(\mathcal{H})+1$.
\end{itemize}
\end{theorem}
Note that, in addition to presenting another proof for Theorem~\ref{alihajijctb}, 
Theorem~\ref{coind} also reveals a new way to compute  
parameters ${\rm coind}(B_0(G))$ and ${\rm coind}(B(G))$ for 
a graph $G$ as well.
Also, it is worth noting that the results in~\cite{2013arXiv1302.5394A, 2014arXiv1401.0138A, MatchingNew} confirm that we can evaluate
the chromatic number of some family of graphs by computing $\alt(-)$ or $\salt(-)$ for an appropriate choice of a hypergraph, while it seems 
that it is~not easy to directly determine 
the amounts of ${\rm coind}(B_0(-))+1$ and
${\rm coind}(B(-))+2$ for these graphs by topological methods.\\

\noindent{\bf Remark.} Motivated by Corollary~\ref{galegencor},  we can assign to any hypergraph $\mathcal{H}$ two topological parameters; the
{\it dimension of $\mathcal{H}$} and the {\it strong dimension of $\mathcal{H}$} denoted by ${\rm dim}(\mathcal{H})$ and ${\rm sdim}(\mathcal{H})$, respectively.
The dimension of $\mathcal{H}$ (resp. strong dimension of $\mathcal{H}$)  is the 
maximum integer $d\geq -1$ such that the vertices of  $\mathcal{H}$ 
can be identified with a multiset $Z\subset S^d$  
so that for any $x\in S^d$, at least one of (resp. both of ) open hemispheres $H(x)$ and $H(-x)$ contains some edge of $\mathcal{H}$. 
In view of Corollary~\ref{galegencor} and with a proof similar to the one of Theorem~\ref{coind}, one can prove the following corollary. 
\begin{corollary}
For a hypergraph $\mathcal{H}$, we have the following inequalities;
\begin{itemize}
\item[{\rm a)}]  ${\rm coind}(B_0(\KG(\mathcal{H})))\geq {\rm dim}(\mathcal{H}) \geq |V(\mathcal{H})|-\alt(\mathcal{H})-1$,
\item[{\rm b)}]  ${\rm coind}(B(\KG(\mathcal{H})))\geq {\rm sdim}(\mathcal{H}) \geq |V(\mathcal{H})|-\salt(\mathcal{H})-1$.
\end{itemize}
\end{corollary}

This paper is organized as follows.
Section~\ref{Preliminaries} contains a brief review of elementary but essential preliminaries and definitions which will be needed throughout the paper. Section~\ref{proofs} is devoted to the proof of Lemma~\ref{galegen} and Theorem~\ref{coind}. Also, in this section, as an application of the generalization of Gale's lemma (Lemma~\ref{galegen}),  we reprove a result by Chen~\cite{JGT21826} about the multichromatic number of stable Kneser graphs.
\section{Topological Preliminaries}\label{Preliminaries}
The following is a brief overview of some topological concepts needed throughout the paper.
We refer the reader to~\cite{MR1988723, MR2279672}  for more basic definitions of algebraic topology.
A $\mathbb{Z}_2$-space is a pair $(T,\nu)$, where $T$ is a topological space
and $\nu$ is an {\it involution}, i.e., $\nu:T\longrightarrow T$ is a continuous map such that $\nu^2$ is the identity map.
For an $x\in T$, two points $x$ and $\nu(x)$ are called {\it antipodal}.
The $\mathbb{Z}_2$-space $(T,\nu)$ is called {\it free} if
there is no $x\in T$ such that $\nu(x)=x$.
For instance, one can see that the unit sphere $S^d\subset \mathbb{R}^{d+1}$ with the involution given by the antipodal map $-:x\rightarrow -x$ is free.
For two $\mathbb{Z}_2$-spaces $(T_1,\nu_1)$ and $(T_2,\nu_2)$,
a continuous map $f:T_1\longrightarrow T_2$ is a $\mathbb{Z}_2$-map if $f\circ\nu_1=\nu_2\circ f$.
The existence of such a map is denoted by $(T_1,\nu_1)\stackrel{\mathbb{Z}_2}{\longrightarrow} (T_2,\nu_2)$.
For a $\mathbb{Z}_2$-space $(T,\nu)$, we define the $\mathbb{Z}_2$-index and $\mathbb{Z}_2$-coindex of $(T,\nu)$, respectively, as follows
$${\rm ind}(T,\nu)=\min\{d\geq 0\:\ (T,\nu)\stackrel{\mathbb{Z}_2}{\longrightarrow} (S^d,-)\}$$
and
$${\rm coind}(T,\nu)=\max\{d\geq 0\:\ (S^d,-)\stackrel{\mathbb{Z}_2}{\longrightarrow} (T,\nu)\}.$$
If for any $d\geq 0$, there is no $(T,\nu)\stackrel{\mathbb{Z}_2}{\longrightarrow}  (S^d,-)$, then we set ${\rm ind}(T,\nu)=\infty$.
Also, if $(T,\nu)$ is~not free, then ${\rm ind}(T,\nu)={\rm coind}(T,\nu)=\infty$.

For simplicity of notation, when the involution is understood from the context, we speak about $T$ rather than the pair $(T,\nu)$; also,
we set ${\rm ind}(T,\nu)={\rm ind}(T)$ and ${\rm coind}(T,\nu)={\rm coind}(T)$.
Throughout this paper,
we endow the unit sphere $S^d\subset \mathbb{R}^{d+1}$ with the involution given by the antipodal map.
Note that if $T_1\stackrel{\mathbb{Z}_2}{\longrightarrow} T_2$, then ${\rm ind}(T_1)\leq {\rm ind}(T_2)$ and
${\rm coind}(T_1)\leq {\rm coind}(T_2)$. Two $\mathbb{Z}_2$-spaces $T_1$ and $T_2$  are
$\mathbb{Z}_2$-equivalent, denoted $T_1\stackrel{\mathbb{Z}_2}{\longleftrightarrow} T_2$, if $T_1\stackrel{\mathbb{Z}_2}{\longrightarrow} T_2$ and $T_2\stackrel{\mathbb{Z}_2}{\longrightarrow} T_1$. In particular, $\mathbb{Z}_2$-equivalent spaces have the same index and also coindex.

In the following, we introduce the concept of simplicial complex
which provides a bridge between combinatorics and topology. A simplicial complex can be viewed as a combinatorial object, called abstract simplicial complex, or as a topological space, called geometric simplicial complex.
Here we just present the definition of an abstract simplicial complex. However, it should be  mentioned that
we can assign a geometric simplicial complex to an abstract simplicial complex, called its geometric realization, and vice versa.
An {\it abstract simplicial complex} is a pair $L=(V,K)$, where $V$ (the vertex set of $L$) is a set  and $K\subseteq 2^V$ (the set of simplicial complexes of $L$)
is a hereditary collection of subsets of $V$, i.e., if $A\in K$ and $B\subseteq A$, then $B\in K$.
Any set $A\in K$ is called a complex of $L$.
The geometric realization of an abstract simplicial complex  $L$ is denoted by $||L||$.
For two abstract simplicial complexes $L_1=(V_1,K_1)$ and $L_2=(V_2,K_2)$,
a simplicial map $f:L_1\longrightarrow L_2$ is map from $V_1$ into $V_2$ which
preserves the complexes, i.e., if $A\in K_1$, then $f(A)\in K_2$.
A simplicial involution is a simplicial map $\nu: L\longrightarrow L$ such that $\nu^2$
is the identity map.
A $\mathbb{Z}_2$-simplicial complex is a pair $(L,\nu)$ where $L$ is a simplicial complex and $\nu:L\longrightarrow L$ is a simplicial involution.
A simplicial involution $\nu$ and a simplicial complex $(L,\nu)$ is free if there is no face $x$ of $L$ such that $\nu(x)=x$.
For two $\mathbb{Z}_2$-simplicial complexes $(L_1,\nu_1)$ and $(L_2,\nu_2)$,
the map $f:L_1\longrightarrow L_2$ is called a
$\mathbb{Z}_2$-simplicial map if $f$ is a simplicial map and $f\circ\nu_1=\nu_2\circ f$.
The existence of a simplicial map ($\mathbb{Z}_2$-simplicial map) $f:L_1\longrightarrow L_2$ implies the existence of a continuous $\mathbb{Z}_2$-map $||f||:||L_1||\stackrel{\mathbb{Z}_2}{\longrightarrow} ||L_2||$
which is called the geometric realization of $f$. If $||L||$ is a ${\mathbb{Z}_2}$-space, we use ${\rm ind}(L)$ and ${\rm coind}(L)$ for ${\rm ind}(||L||)$ and ${\rm coind}(||L||)$, respectively.

The existence of a homomorphism between two graphs is an important and generally challenging problem in graph theory.
In particular, in general, it is a hard task to determine the chromatic number of a graph $G$.
In the following,  we assign some free simplicial $\mathbb{Z}_2$-complexes to graphs
in such a way that graph homomorphisms give rise to $\mathbb{Z}_2$-maps of the corresponding complexes. For a graph $G=(V(G),E(G))$ and
two disjoint subsets $A, B\subseteq V(G)$, define $G[A,B]$ to be the induced bipartite subgraph of $G$ whose parts are $A$ and $B$.\\

\noindent{\bf Box Complex.}
For a graph $G=(V(G),E(G))$ and a subset $A\subseteq V(G)$, set
$${\rm CN}(A)=\{v\in V(G):\ av\in E(G)\ {\rm for\ all\ } a\in A\ \}\subseteq V(G)\setminus A.$$
The {\it box complex} of a graph $G$, $B(G)$, is a free simplicial $\mathbb{Z}_2$- complex with the vertex set $V(G) \uplus V(G) = V(G)\times[2]$ and the following set of simplices 
$$\{A\uplus B:\ A,B\subseteq V(G),\ A\cap B=\varnothing,\ G[A,B]\ {\rm is\ complete,\ and}\ {\rm CN}(A)\neq\varnothing \neq {\rm CN}(B) \}.$$
 Also, one can consider another box complex $B_0(G)$ with the vertex set $V(G) \uplus V(G) = V(G)\times[2]$ and the following set of simplices 
$$\{A\uplus B:\ A,B\subseteq V(G),\ A\cap B=\varnothing,\ G[A,B]\ {\rm is\ complete} \}.$$

An involution  on $B(G)$ (resp. $B_0(G)$) is given by interchanging the two copies of $V(G)$;
that is, $(v,1)\rightarrow(v,2)$ and $(v, 2)\rightarrow(v, 1)$, for any $v \in V (G)$. In view of these involutions, one can consider
$||B(G)||$ and $||B_0(G)||$
as free ${\mathbb Z}_2$-spaces. One can check that any graph homomorphism $G\rightarrow H$ implies that there are
two  simplicial ${\mathbb Z}_2$-maps $B(G)\stackrel{\mathbb{Z}_2}{\longrightarrow} B(H)$ and $B_0(G)\stackrel{\mathbb{Z}_2}{\longrightarrow} B_0(H)$; and consequently,
${\rm ind}(B(G))\leq {\rm ind}(B(H))$, ${\rm coind}(B(G))\leq {\rm coind}(B(H))$,
${\rm ind}(B_0(G))\leq {\rm ind}(B_0(H))$, and ${\rm coind}(B_0(G))\leq {\rm coind}(B_0(H))$.
One can check that $B(K_n)$ and  $B_0(K_n)$ are $\mathbb{Z}_2$-equivalent to $S^{n-2}$ and
$S^{n-1}$, respectively. Hence, $\chi(G)\geq {\rm ind}(B(G))+2\geq {\rm coind}(B(G))+2$ and
$\chi(G)\geq {\rm ind}(B_0(G))+1\geq {\rm coind}(B_0(G))+1$.
Indeed, it is known (see~\cite{MR1988723, MR2279672, MR2452828})
\begin{equation}\label{lbchrom}
\chi(G)\geq {\rm ind}(B(G))+2 \geq {\rm ind}(B_0(G))+1 \geq {\rm coind}(B_0(G))+1 \geq {\rm coind}(B(G))+2.
\end{equation}
\section{Proof of Main Results}\label{proofs}
We should mention that
the following proof is based on an idea similar to that used in an
interesting proof of Ziegler for Gale's lemma~(see page 67 in \cite{MR1988723}).

\noindent{\bf Proof of Lemma~\ref{galegen}.}
For simplicity of notation, assume that $V=\{v_1,\ldots,v_n\}$ where $\sigma(i)=v_i$.
Consider the following curve
$$\gamma=\{(1,t,t^2,\ldots t^{d})\in\mathbb{R}^{d+1}:\ t\in\mathbb{R}\}$$
and set $W=\{w_1,w_2,\ldots,w_n\}$, where $w_i=\gamma(i)$, for $i=1,2,\ldots,n$. Now, let $Z=\{z_1,z_2,\ldots,z_n\}\subseteq S^d$ be a set such that $z_i=(-1)^i{w_i\over ||w_i||}$,
for any $1\leq i\leq n$. Note that if $d\geq 1$, then $Z$ is a set.
Consider the identification  between $V$ and $Z$ such that $v_i\in V$ is identified with $z_i$, for any $1\leq i\leq n$.
It can be checked that every hyperplane of $\mathbb{R}^{d+1}$ passing trough the origin intersects
$\gamma$ in no more than $d$ points.
Moreover, if a hyperplane intersects the curve in exactly $d$ points, then the hyperplane cannot be tangent to the curve; and consequently, at each intersection point, the curve passes from one side of the hyperplane to the other side.

Now, we show that for any $y\in S^d$, $Z_{y}\in {\mathcal P}$.
On the contrary, suppose that there is a $y\in S^d$ such that $Z_{y}\not\in {\mathcal P}$.
Let $h$ be the hyperplane passing trough the origin which contains
the boundary of $H(y)$.
We can move this hyperplane continuously to a position such that it still contains the origin and
has exactly $d$ points of $W=\{w_1,w_2,\ldots,w_n\}$ while during this movement no points of $W$ crosses from one side of $h$ to the other side.
Consequently, during the aforementioned movement, no points of $Z=\{z_1,z_2,\ldots,z_n\}$
crosses from one side of $h$ to the other side.
Hence, at each of these intersections, $\gamma$ passes from one side of $h$ to the other side. 
Let $h^+$ and $h^-$ be two open half-spaces determined by the hyperplane $h$. 
Now consider $X=(x_1,x_2,\ldots,x_n)\in\{+,-,0\}^{n}\setminus\{\zero\}$ such that
$$
x_i=\left\{
\begin{array}{cl}
0 & {\rm if}\ w_i\ {\rm is\ on}\ h\\
+ & {\rm if}\  w_i\ {\rm is\ in }\ h^+\ {\rm and}\ i\ {\rm is\ even}\\
+ & {\rm if}\  w_i\ {\rm is\ in }\ h^-\ {\rm and}\ i\ {\rm is\ odd}\\
- & {\rm otherwise.}
\end{array}.\right.$$
Assume that $x_{i_1},x_{i_2},\ldots,x_{i_{n-d}}$ are nonzero entries of $X$, where $i_1<i_2<\cdots <i_{n-d}$.
It is easy to check that any two consecutive terms of  $x_{i_j}$'s have different signs.
Since $X$ has $n-m=\alt({\mathcal P},\sigma)+1$ nonzero entries, we have $\alt(X)=\alt(-X)=\alt({\mathcal P},\sigma)+1$;
and therefore, both $X_\sigma$ and $(-X)_\sigma$ are in ${\mathcal P}$. Also, one can see that 
either $X_\sigma\subseteq Z_y$ or $(-X)_\sigma\subseteq Z_y$. Therefore, since ${\mathcal P}$ is an signed-increasing property, we have $Z_y\in {\mathcal P}$ which is a contradiction. 
\hfill$\square$\\


\noindent{\bf Multichromatic Number of Stable Kneser Graphs.}
For  positive integers $n,k$, and $s$, the $s$-stable Kneser graph  $\KG(n,k)_s$
is an induced subgraph of $\KG(n,k)$ whose vertex set is  ${[n]\choose k}_s$. In other words, $\KG(n,k)_s=\KG(\left([n],{[n]\choose k}_s\right))$.  The chromatic number of stable Kneser graphs has been studied in several papers~\cite{MR2448565,jonsson,MR2793613}.  Meunier~\cite{MR2793613} posed a conjecture about the chromatic number of stable Kneser hypergraphs which is a generalization of a conjecture of Alon, Drewnowski, and 
{\L}uczak~\cite{MR2448565}.
In the case of graphs instead of hypergraphs, Meunier's conjecture asserts that the chromatic number of $\KG(n,k)_s$ is $n-s(k-1)$ for $n\geq sk$ and $s\geq 2$. Clearly, in view of Schrijver's result~\cite{MR512648}, this conjecture is true for $s=2$.  Moreover, 
for $s\geq 4$ and $n$ sufficiently large,  Jonsson~\cite{jonsson} gave an affirmative answer to the graph case of Meunier's conjecture. 

For two positive integers $m$ and  $n$ with $n\geq m$, 
an $m$-fold $n$-coloring of a graph $G$ is a homomorphism from $G$ to ${\rm KG}(n,m)$.
The $m^{th}$ multichromatic number of a graph $G$, $\chi_m(G)$, is defined as follows 
$$\displaystyle\chi_m(G)=\min\left\{n:\ G\longrightarrow {\rm KG}(n,m)\right\}.$$
Note that $\chi_1(G)=\chi(G)$.
In general, an $m$-fold $n$-coloring of a graph $G$ is called a multicoloring of $G$ with color set $[n]$. 
The following conjecture of Stahl~\cite{Stahl1976185} has received a considerable attention in the literature. 
\begin{alphconjecture}{\rm (\cite{Stahl1976185})}\label{stahlconj}
For positive integers $m, n$, and $k$ with $n\geq 2k$, we have
$$\chi_m({\rm KG}(n,k))=\lceil{m\over k}\rceil(n-2k)+2m.$$
\end{alphconjecture}
Stahl~\cite{Stahl1998287} proved the accuracy of this Conjecture for $k=2,3$ and arbitrary values of $m$. Chen~\cite{JGT21826} studied the multichromatic number of 
$s$-stable Kneser graphs and generalized Schrijver's result. 
In what follows, as an application of Lemma~\ref{galegen}, we present another proof of Chen's result. 
Note that, as a special case of Chen's result, we have the chromatic number of $s$-stable Kneser graphs provided that $s$ is even. 
\begin{alphtheorem}{\rm \cite{JGT21826}}
For positive integers $n, k,$ and $s$ with $n\geq sk$, if $s$ is an even integer and  $k\geq m$, then
$\chi_m({\rm KG}(n,k)_s)=n-sk+sm$.
\end{alphtheorem}
\begin{proof}{
It is straightforward to check that
$\chi_m({\rm KG}(n,k)_s)\leq n-sk+sm$.
For a proof of this observation, we refer the reader to~\cite{JGT21826}.
Therefore, it is enough to show $\chi_m({\rm KG}(n,k)_s)\geq n-sk+sm$.
For the set $[n]$,
let ${\mathcal P}={\mathcal P}(n,k,s)\subseteq P_s([n])$ be a signed-increasing property such that
$(A,B)\in {\mathcal P}$ if each of $A$ and $B$ contains at least ${s\over 2}$ pairwise disjoint
$s$-stable $k$-subsets of $[n]$. One can see that $\alt({\mathcal P},I)=sk-1$ where $I:[n]\longrightarrow [n]$ is the identity bijection. Thus, by Lemma~\ref{galegen}, for $d=n-sk$,  there exists a multiset $Z\subset S^{d}$ of size $n$ such that under a suitable identification of $Z$ with $V$, for any $x\in S^{d}$,  $Z_x\in{\mathcal P}$. In other words, for any $x\in S^{d}$, $H(x)$ contains at least ${s\over 2}$ pairwise disjoint vertices of ${\rm KG}(n,k)_s$.
Now let $c:V({\rm KG}(n,k)_s)\longrightarrow {C\choose m}$ be an $m$-fold  $C$-coloring of ${\rm KG}(n,k)_s$. 
For each $i\in \{1,2,\ldots,C-sm+1\}$, define
$A_i$ to be a set consisting of all $x\in S^d$ such that $H(x)$ contains some vertex with color $i$.

Furthermore, define $A_{C-sk+2}=S^d\setminus \displaystyle\cup_{i=1}^{C-sm+1}A_i$.
One can check that each $A_i$ contains no pair of
antipodal points, i.e., $A_i\cap (-A_i)=\varnothing$; and also, for $i\in\{1,2,\ldots, C-sm+1\}$,
$A_i$ is an open subset of $S^d$ and $A_{C-sm+2}$
is closed. Now by the Borsuk-Ulam theorem, i.e., for any covering of $S^{d}$ by $d+1$ sets $B_1,\ldots, B_{d+1}$, each $B_i$ open or closed, there exists an 
$i$ such that $B_i$ contains a pair of antipodal points. Accordingly, we have $C-sm+2\geq d+2=n-sk+2$ which completes the proof.
}\end{proof}


In what follows, in view of Lemma~\ref{galegen} and with a similar approach  as in proof of Proposition~8 of~\cite{MR2279672}, we prove Theorem~\ref{coind}.

\noindent{\bf Proof of Theorem~\ref{coind}.}
Let $\sigma:[n]\longrightarrow V(\mathcal{H})$ be an arbitrary  bijection.
To prove the first part, set
$d=|V|-\alt({\mathcal H},\sigma)-1$. In view of inequalities of (\ref{lbchrom}) and the definition of $\alt({\mathcal H})$, it is sufficient to prove that
${\rm coind}(B_0(G))+1\geq |V|-\alt({\mathcal H},\sigma)$. If $d\leq 0$, then one can see that the assertion follows.
Hence, suppose $d\geq 1$. Now in view of Corollary~\ref{galegencor}, there exists an $n$-set $Z\subset S^d$ and an identification of $Z$ with $V$ such that for any $x\in S^{d}$, at least one of open hemispheres $H(x)$ and $H(-x)$ contains some edge of ${\mathcal H}$.

For any vertex $A$ of ${\rm KG}({\mathcal H})$ and any $x\in S^d$, define $D_A(x)$ to be
the smallest distance of a point in $A\subset S^d$ from the set $S^d\setminus H(x)$.
Note that $D_A(x)>0$ if and only if $H(x)$ contains $A$.
Define
$$D(x)=\displaystyle\sum_{A\in E}\left(D_A(x)+D_A(-x)\right).$$
Since, 
for any $x\in S^d$, at least one of $H(x)$ and $H(-x)$ contains some edge of ${\mathcal H}$, 
we have $D(x)>0$.
Thus, the map
$$f(x)= {1\over D(x)}\left(\sum_{A\in E}D_A(x)||(A,1)||+\sum_{A\in E}D_A(-x)||(A,2)||\right)$$
is a $\mathbb{Z}_2$-map from $S^d$ to $||B_0({\rm KG}({\mathcal H}))||$. It implies ${\rm coind}(B_0(G))\geq d$.\\

\noindent {\rm b)} To prove the second part,  set $d=n-\salt({\mathcal H},\sigma)-1$. In view of inequalities of (\ref{lbchrom}) and the definition of $\salt({\mathcal H})$, 
it is sufficient to prove that
${\rm coind}(B(G))+2\geq |V|-\salt({\mathcal H},\sigma)+1$. If $d\leq 0$, then one can see that the assertion follows.
Hence, suppose $d\geq 1$. Now
in view of Corollary~\ref{galegencor}, there is an $n$-set $Z\subset S^d$ and an identification of $Z$ with $V$ such that for any $x\in S^{d}$,
$H(x)$ contains some edge of ${\mathcal H}$. Define
$D(x)=\displaystyle\sum_{A\in E}D_A(x).$
For any $x\in S^d$, $H(x)$ has some edge of ${\mathcal H}$ which implies that $D(x)>0$.
Thus, the map
$$f(x)= {1\over 2D(x)}\sum_{A\in E}D_A(x)||(A,1)||+{1\over 2D(-x)}\sum_{A\in E}D_A(-x)||(A,2)||$$ 
is a $\mathbb{Z}_2$-map from $S^d$ to $||B({\rm KG}({\mathcal H}))||$. This implies ${\rm coind}(B(G))\geq d$.
\hfill$\square$\\

\noindent{\bf Acknowledgement:}
The authors would like to express their deepest gratitude to Professor Carsten~Thomassen for his insightful comments.
They also appreciate the detailed valuable comments of  Dr.~Saeed~Shaebani.
The research of Hossein Hajiabolhassan is partially supported by ERC advanced grant GRACOL. A part of this paper was written while Hossein Hajiabolhassan was visiting School of Mathematics, Institute for Research in Fundamental Sciences~(IPM). He acknowledges the support of IPM.
Moreover, they would like to thank Skype for sponsoring their endless conversations in two countries.


\def\cprime{$'$} \def\cprime{$'$}

\end{document}